\documentclass[12pt]{amsart}
\usepackage{amsmath}
\usepackage{amssymb}
\usepackage{epsfig}
\usepackage{graphicx}
\numberwithin{equation}{section}

\input xy
\xyoption{all}

\calclayout
\allowdisplaybreaks[3]

\theoremstyle{plain}
\newtheorem{prop}{Proposition}

\newtheorem{coro}[prop]{Corollary}

\theoremstyle{definition}
\newtheorem{defi}[prop]{Definition}

\newtheorem{ques}[prop]{Question}

\newtheorem{rema}[prop]{Remark}

\newtheorem{exam}[prop]{Example}

\def\ra{\rightarrow}

\def\cC{{\mathcal C}}

\def\cE{{\mathcal E}}

\def\cI{{\mathcal I}}

\def\cM{{\mathcal M}}

\def\cO{{\mathcal O}}

\def\cQ{{\mathcal Q}}
\def\cR{{\mathcal R}}
\def\cS{{\mathcal S}}

\def\cW{{\mathcal W}}
\def\cX{{\mathcal X}}
\def\cY{{\mathcal Y}}

\def\oM{{\overline{M}}}
\def\ocM{{\overline{\mathcal M}}}

\def\fS{{\mathfrak S}}

\def\sg{{\mathsf g}}

\def\bG{{\mathbb G}}
\def\bP{{\mathbb P}}

\def\bZ{{\mathbb Z}}

\def\bC{{\mathbb C}}

\def\Bl{\mathrm{Bl}}
\def\Pic{\mathrm{Pic}}

\def\Gr{\mathrm{Gr}}
\def\GIT{/\!\!/}
\def\SL{\mathrm{SL}}
\def\PGL{\mathrm{PGL}}

\def\Hom{\mathrm{Hom}}

\def\lim{\mathrm{lim}}

\def\Prym{\mathrm{Prym}}

\def\Sym{\mathrm{Sym}}
\def\IJ{\mathrm{IJ}}
\def\JJ{\mathrm{J}}

\makeatother
\makeatletter

\author{Brendan Hassett}
\address{Department of Mathematics\\
Rice University, MS 136 \\
Houston, TX  77251-1892 \\
USA}
\email{hassett@rice.edu}

\author{Yuri Tschinkel}
\address{Courant Institute\\
                New York University \\
                New York, NY 10012 \\
                USA }
\email{tschinkel@cims.nyu.edu}

\address{Simons Foundation\\
160 Fifth Avenue\\
New York, NY 10010\\
USA}

\title{Quartic del Pezzo surfaces over function fields of curves}

\begin{document}
\date{\today}

\maketitle

\section{Introduction}

The geometry of spaces of rational curves of low degree on Fano threefolds is a very active
area of algebraic geometry.  One of the main goals is to understand the Abel-Jacobi morphism 
from the base of a family 
of such curves to the intermediate Jacobian of the threefold.  For instance, does a given family of curves
dominate the intermediate Jacobian?  If so, what are the fibers of this morphism?  Does it
give the maximal rationally connected (MRC) quotient?  
Representative results in this direction are available for:
\begin{itemize}
\item{cubic threefolds \cite{HRS,HRS2,IM,MT};}
\item{Fano threefolds of genus six and degree $10$ in $\bP^7$ \cite{DIM};}
\item{Fano threefolds of genus seven and degree $12$ in $\bP^8$ \cite{IM2};}
\item{moduli of vector bundles \cite{Castravet}---this case makes clear that
one cannot always expect the morphism to the intermediate 
Jacobian to give the MRC fibration.}
\end{itemize}

At the same time, del Pezzo fibrations 
$\pi:\cX \ra \bP^1$ are equally interesting
geometrically.  Moreover, the special case where the rational curves
happen to be {\em sections} of $\pi$ is particularly important
for arithmetic applications.  It is a major open problem to determine
whether or not sections exist
over a non-closed ground field.  Of course, the Tsen-Lang theorem
gives sections when the ground field is algebraically closed.
Even when there are rigid sections,
these are typically defined over extensions of large degree.  

However, suppose that the space of sections 
of fixed height is rationally connected over
the intermediate Jacobian $\IJ(\cX)$.  If the fibration
is defined over a finite field and the space of sections descends
to this field then $\pi$ has sections over that field.  
Indeed, this follows by combining a theorem of Lang \cite{Lang}
(principal homogeneous spaces
for abelian varities over finite fields are trivial) and
a theorem of Esnault \cite{Esn} (rationally connected varieties
over finite fields have rational points).  This point of view
was developed for quadric surface bundles in \cite{HTquad},
with applications to effective versions of weak approximation.

This paper addresses the case of quartic del Pezzo fibrations
$\pi:\cX \ra \bP^1$.  Throughout, we assume $\cX$ is smooth
and the degenerate fibers of $\pi$ have at worst one ordinary singularity.
We classify numerical invariants of the family---the fundamental
invariant is the {\em height}, denoted $h(\cX)$.  We provide
explicit geometric realizations for fibrations of small height.
We exhibit families of
sections parametrizing the intermediate Jacobian and admitting
rationally-connected fibrations over the intermediate Jacobian.
The underlying constructions often involve Brill-Noether theory
for smooth or nodal curves.  
Our main result is the comprehensive analysis of spaces of small-height sections
for height $12$ fibrations in Section~\ref{sect:twelve}; here $\cX$ dominates a 
{\em nodal} Fano threefold of genus seven and degree $12$.

\

{\bf Acknowledgments:} 
The first author was supported by National Science Foundation Grants 0901645, 0968349, and 1148609;
the second author was supported by National Science Foundation
Grants 0739380, 0968318, and 1160859.
We are grateful to Asher Auel, Jean-Louis Colliot-Th\'el\`ene, and
R.~Parimala for useful conversations and Nikita Kozin for comments
on the manuscript; we appreciate
Jason Starr and Yi Zhu explaining the results of
\cite{Zhu} on Abel-Jacobi morphisms for fibrations
of rationally connected varieties over curves.

\section{Basic properties of quartic del Pezzo surfaces}  
Let $X$ be a quartic del Pezzo surface over an algebraically closed field.
$X$ is isomorphic to $\bP^2$ blown up at five distinct points
such that no three are collinear.  Thus we may identify
$$\Pic(X)=\bZ L + \bZ E_1 + \bZ E_2 + \bZ E_3 + \bZ E_4 + \bZ E_5,$$
where $L$ is the pull-back of the hyperplane class on $\bP^2$ and the $E_i$
are the exceptional divisors.  
The primitive divisors
$$\Lambda=K_X^{\perp}\subset \Pic(X)$$
are a lattice under the intersection form; using the basis
$$\{ E_1-E_2,E_2-E_3,E_3-E_4,E_4-E_5,L-E_1-E_2-E_3 \}$$
we have
$$\Lambda \simeq \left( \begin{matrix} 
                        -2 & 1 & 0 & 0 & 0 \\
                        1  & -2 & 1 & 0 & 0\\
                        0  & 1  & -2 & 1 & 1 \\
                        0  & 0  & 1  & -2 & 0 \\
                        0  & 0  & 1  & 0  & -2
                \end{matrix} \right).
$$
Writing $\Lambda^*=\Hom(\Lambda,\bZ)$, the intersection form gives a 
surjection
$$\Pic(X) \twoheadrightarrow \Lambda^*,$$
with kernel generated by $K_X$.
Note that $\Lambda^*/\Lambda$ is a cyclic group of order four. 

The action of the monodromy group on $\Pic(X)$ factors through the Weyl group
$W(D_5)$ coming from the $D_5$ root system contained in the lattice $(-1)\Lambda$.
Abstractly, $W(D_5)$ may be realized as an extension of the symmetric group
\begin{equation} \label{eqn:ext}
1 \ra (\bZ/2\bZ)^4 \ra W(D_5) \ra \fS_5 \ra 1;
\end{equation}
concretely, this is the subgroup of the group of {\em signed} 
$5\times 5$ permutation matrices
where the determinant of the matrix equals the sign of the permutation.

The anticanonical model of $X$ is a complete
intersection of two quadrics 
$$X=\{Q_0=Q_1=0 \} \subset \bP^4.$$
Let $\Gr(2,\Gamma(\cO_{\bP^4}(2)))\simeq \Gr(2,15)$ denote the Grassmannian
parametrizing pencils of quadrics in $\bP^4$.  The group $\SL_5$ acts naturally on
$\bP^4$ and this Grassmannian, linearized
via the Pl\"ucker embedding.
Consider the loci of stable and semistable points
$$U_s \subset U_{ss} \subset \Gr(2,\Gamma(\cO_{\bP^4}(2)))$$
and the resulting GIT quotient
$$\oM:=\Gr(2,\Gamma(\cO_{\bP^4}(2))) \GIT \SL_5.$$

Principal results of \cite{MM} include:
\begin{itemize}
\item{$U_s$ coincides with the locus of smooth quartic del Pezzo surfaces;}
\item{$U_{ss}$ parametrizes quartic del Pezzo surfaces with at worst ordinary double points;}
\item{smooth quartic del Pezzo surfaces are {\em diagonalizable}, i.e., there
exist coordinates $x_0,\ldots,x_4$ on $\bP^4$ such that
$$X=\{x_0^2+x_1^2+x_2^2+x_3^2+x_4^2=c_0x_0^2+c_1x_1^2+c_2x_2^2+c_3x_3^2+c_4x_4^2=0\};$$}
\item{closed orbits in $U_{ss}$ correspond to diagonalizable del Pezzo surfaces,
which have $0,2$, or $4$ ordinary double points;}
\item{the determinantal scheme of a quartic del Pezzo surface
$$X \mapsto Z=\{\det(s_0Q_0+s_1Q_0)=0\} \subset \bP^1_{[s_0,s_1]}= \bP(\Gamma(\cI_X(2)))$$
induces a rational map
\begin{equation} \label{eqn:det}
\Xi:\Gr(2,\Gamma(\cO_{\bP^4}(2))) \dashrightarrow  \bP(\Gamma(\cO_{\bP^1}(5)))\GIT \SL_2,
\end{equation}
regular on $U_{ss}$.
Consequently, we obtain an isomorphism of projective schemes
$$\oM \simeq \bP(\Gamma(\cO_{\bP^1}(5))) \GIT \SL_2 \simeq \oM_{0,5}/\fS_5,$$
where $\oM_{0,5}$ is the moduli space of stable rational curves with five marked
points, with the canonical action of the symmetric group on five letters.}
\end{itemize}

The diagonal representation of a smooth quartic del Pezzo surface
has useful implications:
\begin{quote}
Consider the group $H=(\bZ/2\bZ)^4$, realized as $5\times 5$ diagonal
matrices with $\pm 1$ on the diagonal modulo the subgroup $\pm I$.
Then $H$ acts on $X$ and this action is transitive on 
the $16$ lines of $X$.
\end{quote}
The closure in $U_{ss}$ of the
orbit of any quartic del Pezzo surface with 
a unique ordinary double point contains quartic del Pezzo surfaces with
at least {\em two} ordinary double points.  Indeed, quartic del Pezzo surfaces 
with one ordinary double point are strictly semistable but have orbits
that are not closed.  
In particular, the moduli stack of quartic del Pezzo surfaces with
ordinary singularities is non-separated.
Concretely, fix a quartic del Pezzo surface $X=\{Q_0=Q_1=0\}$ with one ordinary singularity.  
Supposing
the singularity is at $[0,0,0,0,1]$
with tangent space $x_0=0$, 
and also that $Q_0$ cuts out this tangent space and $Q_1$ is nodal at $[0,0,0,0,1]$,
we obtain
$$Q_0=x_0x_4+Q_2(x_0,x_1,x_2,x_3), \quad  Q_1=Q_1(x_0,x_1,x_2,x_3).
$$
Writing $Q_2=x_0 \ell_2(x_0,x_1,x_2,x_3) +R_2(x_1,x_2,x_3)$
and substituting $x_4=x'_4-\ell_2$, we obtain
$$X=\{x'_4 x_0 + R_2(x_1,x_2,x_3)=Q_1=0\}.$$
Writing $Q_1=x_0 \ell_1(x_0,x_1,x_2,x_3)+R_1(x_1,x_2,x_3)$ yields
$$X=\{x'_4 x_0 + R_2(x_1,x_2,x_3)=
x_0\ell_1+R_1(x_1,x_2,x_3)=0\}.$$
Consider the action of the one-parameter subgroup $\rho(t)$ given by
$$
x_0 \mapsto t x_0, \  x_i \mapsto x_i, i=1,2,3, \ 
x'_4 \mapsto t^{-1}x'_4.$$
The limit
$$X_0:=\lim_{t\ra 0}\rho(t)\cdot X=\{
x'_4 x_0 + R_2(x_1,x_2,x_3)=R_1(x_1,x_2,x_3)=0\}$$
has double points $[0,0,0,0,1]$ and $[1,0,0,0,0]$.  
Moreover, when $\{R_1=R_2=0 \} \subset \bP^2$ consists of four
distinct points  we have
$$\mathrm{Aut}^{\circ}(X_0)\simeq \bG_m,$$
i.e., the identity component of the automorphism group
corresponds to $\rho(t)$.

It is a fundamental result of Geometric Invariant Theory 
that the failure of separatedness is governed by
specializations under one-parameter subgroups arising as automorphisms
of objects with closed orbit
(see \cite[Prop.~4.2]{HH}).
Note that there is a unique non-closed orbit in the basin of attraction 
of $X_0$ under $\rho(t)$.   
Thus removing the semistable objects of type $X_0$
eliminates the non-separatedness.  Alternately, we could apply the Kirwan
blow-up process \cite{Kirwan} to eliminate semistable 
points with positive-dimensional
stabilizers.  

We summarize this discussion as follows:
\begin{prop} \label{prop:stacky}
Let $\cM$ denote the moduli stack of quartic del Pezzo surfaces with at worst
one ordinary double point, with coarse moduli space $M$.  Then we have
\begin{itemize}
\item{$\Xi$ (defined in (\ref{eqn:det})) induces a morphism of stacks
$$\xi:\cM \ra [(\ocM_{0,5}\setminus S)/\fS_{5}]$$
where $S$ is the zero-dimensional boundary stratum (see Figure~\ref{fig:FivePoints});}
\item{$\xi$ induces an isomorphism of coarse moduli spaces
$$M \stackrel{\sim}{\ra} (\oM_{0,5}\setminus S)/\fS_{5}.$$}
\end{itemize}
\end{prop}
\begin{figure} 
\includegraphics[scale=.7]{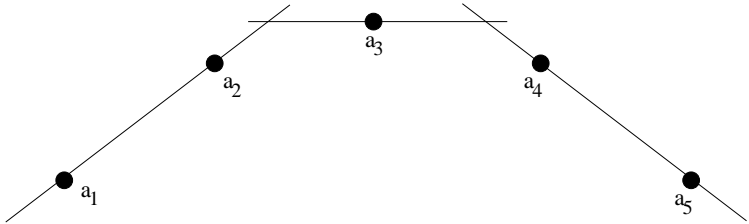}
\caption{The zero-dimensional boundary stratum}
\label{fig:FivePoints}
\end{figure}

The invariants of binary quintics were computed in the 19th century; see
\cite{Baker} for a description convenient for our purposes.  
It follows that 
$$M\simeq \bP(1,2,3)\setminus s$$
where $s$ is the image of $S$, and a point where the minimal degree invariant is non-zero.  
Thus we obtain
\begin{coro}  \label{coro:invariant}
$\Pic(M) \simeq \bZ$.
\end{coro}

\section{Quartic del Pezzo fiber spaces over $\bP^1$}

We work over an algebraically closed field of characteristic different from two.
\begin{prop} \label{prop:hodge}
Let $\cX$ be a smooth projective threefold and $\pi:\cX \ra \bP^1$ 
a morphism such that the generic fiber is a 
smooth rational surface.  Then
\begin{itemize}
\item{$\cX$ is separably rationally connected;}
\item{the Hodge numbers $h^{ij}(\cX)=0$ for $(i,j)\neq (0,0)$, $(1,1)$, $(2,2)$, $(3,3)$, $(1,2)$, $(2,1)$.}
\end{itemize}
\end{prop}
\begin{proof}
$\pi$ admits a section $s:\bP^1 \ra \cX$.  This follows from general results \cite{dJS} or 
by the classification of rational surfaces and the Tsen-Lang theorem (cf.~\cite{HasCMI}).
Since the smooth fibers of $\pi$ are rational surfaces, there exist very
free curves in the fibers incident to $s(\bP^1)$.   The techniques of \cite{GHS}
thus yield a very free section of $\pi$, proving separably rational connectedness.

It follows \cite[IV.3.8]{KollarRC} that
$$\Gamma(\Omega^1_{\cX})=\Gamma(\Omega^2_{\cX})=\Gamma(\Omega^3_{\cX})=0$$
and by Serre duality
$$H^3(\Omega^2_{\cX})=H^3(\Omega^1_{\cX})=H^3(\cO_{\cX})=0.$$
The Leray spectral sequence for $\pi$ and Serre duality also give
$$H^1(\cO_{\cX})=H^2(\cO_{\cX})=0 \quad
H^2(\Omega^3_{\cX})=H^1(\Omega^3_{\cX})=0.$$
\end{proof}

Assume that $\pi:\cX \ra \bP^1$ is a fibration of quartic del Pezzo surfaces,
with $\cX$ smooth.
The fibration has {\em square-free} discriminant if
the fibers are complete intersections of two quadrics
with at worst one ordinary double point; we assume this from now on.  
Under these assumptions, the relative dualizing sheaf $\omega_{\pi}$
is invertible, the direct image of its dual
$$\pi_*\omega_{\pi}^{-1}$$
is locally free, and the higher direct images vanish.  Cohomology-and-base-change
and the fact that $\omega_{\pi}^{-1}$ is very ample on fibers of $\pi$ yield
an embedding $\cX \subset \bP((\pi_*\omega_{\pi}^{-1})^{\vee})$.
Let $\cI_{\cX}(2)$ denote the ideal sheaf of $\cX$
twisted by $\cO_{\bP((\pi_*\omega_{\pi}^{-1})^{\vee})}(2)$;
then $\pi_*(\cI_{\cX}(2)) \subset \Sym^2 (\pi_*\omega_{\pi}^{-1})$
is a subbundle of rank two with quotient $\pi_*\omega_{\pi}^{-2}$.

\begin{prop} \label{prop:liftable}
Let $\pi:\cX \ra \bP^1$ be a fibration of quartic
del Pezzo surfaces with square-free discriminant,
defined over a field of characteristic $p>2$.  Then 
$\cX$ and $\pi$ are liftable to characteristic
zero.  
\end{prop}
\begin{rema} \label{rema:HdR}
Assuming Proposition~\ref{prop:liftable}, the Hodge-de Rham
spectral sequence for $\cX$ degenerates provided $p>3$ (or $=0$)
\cite{DI}.  Thus
$h^{11}(\cX)$ equals the
Picard rank of $\cX$, and $h^{11}$ and $h^{12}$ are unchanged 
on reduction to $\cX$ from characteristic zero.  
\end{rema}
\begin{proof}
Every vector bundle over $\bP^1$ decomposes as a direct
sum of line bundles
$$V\simeq \oplus_{i=1}^r \cO_{\bP^1}(a_i).$$
Clearly any such direct sum lifts to characteristic
zero, i.e., lift each summand individually.
Cohomology commutes with reduction modulo $p$
for such a lift; given lifts of two such bundles,
taking homomorphisms between them commutes with reduction as well.  

In particular, we may lift the decomposition
of $V=\pi_*\omega_{\pi}^{-1}$ to characteristic zero;
the same holds true for $\Sym^2 V$
and its quotient
$V'=\pi_*\omega_{\pi}^{-2}$.  There is a combinatorial
criterion \cite{Shatz}
for which vector bundles arise as subbundles
of a given bundle, independent of the characteristic.
In particular, the quotient homomorphism
$$ \Sym^2 V \twoheadrightarrow V'$$
lifts to characteristic zero, as does its kernel.  
The resulting rank-two subbundle gives the defining 
equations of the lift of $\cX$ to characteristic zero.
Having square-free discriminant is an open condition,
so the resulting lift shares this property.
\end{proof}

We define
$$
h(\cX):=\deg (c_1(\omega_{\pi})^3)
$$
to be the {\em height} of the del Pezzo surface fibration.  
Corollary~\ref{coro:invariant} implies this is the unique numerical 
invariant for families with square-free discriminant:
\begin{prop} \label{prop:numerology}
Let $\pi:\cX \ra \bP^1$ be a quartic del Pezzo fibration with square-free
discriminant.
Then we have:
\begin{itemize}
\item{$h(\cX)=-2\deg(\pi_*\omega^{-1}_{\pi})$;}
\item{
$\Delta:=\text{number of singular fibers of }\pi=2h(\cX)$;}
\item{the topological Euler characteristic
$$\begin{array}{rcl}
\chi(\cX)&=&(2-\Delta)\cdot \chi(\text{smooth fiber})+\Delta\cdot \chi(\text{singular fiber})\\
&=&16-2h(\cX);
\end{array}
$$}
\item{
$\chi(\Omega^1_{\cX})=1-\chi(\cX)/2=h(\cX)-7$;}
\item{
$\text{expected number of parameters}=-\chi(T_{\cX})=\frac{3}{2}h(\cX)-1$.}
\end{itemize}
\end{prop}
\begin{proof}
The quantities $h(\cX), \deg(\pi_*\omega^{-1}_{\pi}),$ and $\Delta$
all depend on the degree of the classifying map $\mu:\bP^1 \ra \cM$,
thus are in constant proportion.  To evaluate the proportionality
constant, it suffices to consider a singular non-constant example,
e.g., the intersection of a quadric in $\bP^4$ with a pencil of quadric
hypersurfaces.  (See Case 1 of Section~\ref{sect:explicit} for details.)

Once we know the number of singular fibers, it is straightforward
to compute the (\'etale) topological Euler characteristic.  The Euler characteristics
of $\Omega^1_{\cX}$ and $\Omega^2_{\cX}$ are
related to the topological Euler characteristic via the Hodge-de Rham 
spectral sequence
$$\chi(\cX)=\chi(\cO_X)-\chi(\Omega^1_{\cX})+\chi(\Omega^2_{\cX})-
\chi(\Omega^3_{\cX});$$
Serre duality and Proposition~\ref{prop:hodge} give
$$\chi(\cX)=2-2\chi(\Omega^1_{\cX}).$$
Finally, the expected number
of parameters $-\chi(T_{\cX})$, which may be computed by Hirzebruch-Riemann-Roch
and Gauss-Bonnet:
$$\begin{array}{rcl}
\chi(T_{\cX})&=&\frac{1}{2} c_1(T_{\cX})^3 - \frac{19}{24} c_1(T_{\cX})c_2(T_{\cX})+ \frac{1}{2}c_3(T_{\cX})\\
             &=&-\frac{1}{2}c_1(\omega_{\cX})^3 - 19 \chi(\cO_{\cX})+\frac{1}{2}\chi(\cX). 
\end{array}
$$
We have 
$$c_1(\omega_{\cX})^3=(c_1(\pi^*\omega_{\bP^1})+c_1(\omega_{\pi}))^3=h(\cX)-24$$
whence
$$\chi(T_{\cX})=-\frac{1}{2}h(\cX)+12-19+8-h(\cX)=-\frac{3}{2}h(\cX)+1,$$
as claimed.
\end{proof}

Combining with Proposition~\ref{prop:hodge} and
Remark~\ref{rema:HdR}, we obtain
\begin{coro}
Retain the notation and assumptions of Proposition~\ref{prop:numerology};
assume the characteristic $p>3$ (or $=0$) 
and set 
$$h^{11}=\dim H^1(\Omega^1_{\cX})=\mathrm{rank}(\Pic(\cX)).$$
Then the Hodge diamond of $\cX$ takes the form
$$\begin{array}{ccccccc}
& & & 1 & & &\\
& &0 &  & 0 & &\\
&0 & & h^{11} &  & 0 &\\
0\quad & & h(\cX)+h^{11}-7 &  & h(\cX)+h^{11}-7  & &\quad 0
\end{array}
$$
\end{coro}

This can be interpreted in geometric terms:
Let $\pi:\cX \ra \bP^1$ be a fibration of quartic del Pezzo surfaces
with square-free discriminant.  Let $D\ra \bP^1$ be the degree-five
covering parametrizing nodal quadric hypersurfaces containing
the fibers of $\pi$, and $\tilde{D} \ra D$ the double covering parametrizing
rulings of these quadric hypersurfaces.  The assumption on singular
fibers implies that $D\ra \bP^1$ is simply branched and $\tilde{D} \ra D$
is \'etale, and we may consider the resulting Prym variety
$\Prym(\tilde{D} \ra D)$.
\begin{prop}
Retaining the notation above, let $A^2(\cX)$ denote codimension two cycles
algebraically equivalent to zero.  Then $A^2(\cX)\simeq \Prym(\tilde{D} \ra D)$. 
\end{prop}
\begin{proof} Fix a section $s:\bP^1 \ra \cX$ of $\pi$, not passing through
any lines of the singular fibers.  Projecting from 
$s$ realizes the generic fiber of $\pi$ as a conic bundle over $\bP^1$.
Indeed, blowing up a quartic del Pezzo surface
at a point not on a line yields a cubic surface with a line; projecting
the cubic surface from that line gives a conic bundle over $\bP^1$.  
There are five degenerate fibers, corresponding to the generic fiber
of $D\ra \bP^1$, and ten irreducible components in these fibers,
corresponding to the the generic fiber of $\tilde{D} \ra \bP^1$.  Thus
the blow-up of $\cX$ along $s(\bP^1)$ is a conic bundle over a Hirzebruch surface
surface, with discriminant curve $D$ and \'etale cover $\tilde{D} \ra D$.
A straightforward variation on the
argument of \cite[\S 3]{BeauvillePrym} implies that
$A^2(\Bl_{s(\bP^1)}(\cX)) \simeq \Prym(\tilde{D} \ra D);$ the same holds for
$A^2(\cX)$.  (For related results, see \cite{Kanev}.)  
\end{proof}

\begin{rema}
Over $\bC$, this Prym variety is the intermediate Jacobian $\IJ(\cX)$;
its construction is compatible with reduction modulo $p$, assuming the 
reduction still has square-free discriminant.  
Even over fields of characteristic $p>0$, $\Prym(\tilde{D} \ra D)$ 
has the requisite universal 
property for one-cycles \cite[\S 3.2]{BeauvillePrym}:  Given a family of
curves
$$\begin{array}{ccc}
\cC & \ra & \cX \\
 \downarrow &  & \\
   T &  & 
\end{array}
$$
over a smooth connected base $T$, there is an Abel-Jacobi
morphism $T\ra \Prym(\tilde{D} \ra D)$.  
\end{rema}

The Lefschetz hyperplane theorem gives:
\begin{prop} \label{prop:monodromy}
Let $\cX \ra \bP^1$ be a fibration of quartic del Pezzo surfaces.
Suppose $\cX \subset \bP((\pi_*\omega_{\pi}^{-1})^{\vee})$ is a smooth complete intersection of 
ample divisors in $\bP((\pi_*\omega_{\pi}^{-1})^{\vee})$.  Then 
$\Pic(\cX)=\bZ^2$.
\end{prop}

\begin{defi}
Let $\pi:\cX \ra \bP^1$ be a quartic del Pezzo fibration with
square-free discriminant.  The {\em anticanonical height} of a section $\sigma:\bP^1 \ra \cX$
is defined 
$$h_{\omega_{\pi}^{-1}}(\sigma)=\deg(\sigma^*\omega^{-1}_{\pi}).$$
\end{defi}
Note that this equals the degree of the normal bundle $N_{\sigma}$, thus the Hilbert
scheme of sections has dimension $\ge h_{\omega_{\pi}^{-1}}(\sigma)+2$ at $\sigma$.

\section{Explicit realization of generic families}
\label{sect:explicit}

In this section, we give concrete geometric realizations of families of each height.
We start by writing
$$
\begin{array}{c}\pi_*\omega_{\pi}^{-1}=\cO_{\bP^1}(-a_1) \oplus
\cO_{\bP^1}(-a_2) \oplus
\cO_{\bP^1}(-a_3) \oplus
\cO_{\bP^1}(-a_4) \oplus
\cO_{\bP^1}(-a_5), \\
a_1 \le a_2 \le a_3 \le a_4 \le a_5.
\end{array}
$$
We refer the reader to \cite[2.2]{Shramov} for a discussion of which $a_j$ may occur.
Here we assume this sheaf is generic, in the sense that it is as
`balanced' as possible, i.e., $a_5-a_1 \le 1$.  Thus 
$\pi_*\omega_{\pi}^{-1}$ is one of the following:
\begin{enumerate}
\item{$\cO_{\bP^1}^5 \otimes \cO_{\bP^1}(-m)$}
\item{$(\cO_{\bP^1}\oplus \cO_{\bP^1}(-1)^4) \otimes \cO_{\bP^1}(-m)$}
\item{$(\cO_{\bP^1}^2\oplus \cO_{\bP^1}(-1)^3) \otimes \cO_{\bP^1}(-m)$}
\item{$(\cO_{\bP^1}^3\oplus \cO_{\bP^1}(-1)^2) \otimes \cO_{\bP^1}(-m)$}
\item{$(\cO_{\bP^1}^4\oplus \cO_{\bP^1}(-1)) \otimes \cO_{\bP^1}(-m)$}
\end{enumerate}
where $m\in \bZ$.  In this situation, Proposition~\ref{prop:numerology}
may be proven by direct computations with resolutions.

\subsection*{Case 1}
We compute invariants for generic quartic del Pezzo surfaces
$$\pi:\cX \ra \bP^1$$
arising from morphisms into the Hilbert scheme
$$
\mu:\bP^1 \ra \Gr(2,\Gamma(\cO_{\bP^4}(2)))
$$
of degree $d\ge 0$, where $\mu$ is
an embedding transverse to the discriminant.  
In this situation, we may embed
$$
\cX \subset P:=\bP^1 \times \bP^4,
$$
where 
\begin{equation} \label{eq:bundle}
\pi_* I_{\cX}(2) \subset \Gamma(\cO_{\bP^4}(2))\otimes \cO_{\bP^1}
\end{equation}
is a rank-two vector subbundle of degree $-d$.  

\subsubsection*{Even part}
Assuming this bundle is generic and $d=2n$, we have
$$
\pi_* I_{\cX}(2) \simeq \cO_{\bP^1}(-n)^2
$$
and $\cX$ is a complete intersection of two forms of bidegree $(n,2)$.
As such, $\cX$ depends on
$$2(15(n+1)-2)-(3+24)=30n-1$$
parameters.  
In this situation
$$\omega^{-1}_{\pi}=\cO_{\cX}(-2n,1), \quad \pi_*\omega^{-1}_{\pi}=\cO_{\bP^1}(-2n)\otimes \Gamma(\cO_{\bP^4}(1))$$
whence 
$$\deg(\pi_*\omega^{-1}_{\pi})=-10n, \quad h(\cX)=20n.$$
We compute the middle cohomology using the exact sequence
$$
0 \ra I_{\cX}/I_{\cX}^2 \ra \Omega^1_{P}|\cX \ra \Omega^1_{\cX} \ra 0,
$$
where the conormal bundle 
$$I_{\cX}/I_{\cX}^2=\cO_{\cX}(-n,-2)^2.$$
Thus we find
$$\chi(\Omega^1_{\cX})=\chi(\Omega^1_P|{\cX})-2\chi(\cO_{\cX}(-n,-2)).$$
Using the Koszul resolution
$$0 \ra \cO_{P}(-2n,-4) \ra \cO_{P}(-n,-2)^2 \ra I_{\cX} \ra 0$$
we compute 
$$\begin{array}{rcl}
\chi(\cO_{\cX}(a,b))&=&\chi(\cO_P(a,b))-2\chi(\cO_P(a-n,b-2))+\chi(\cO_P(a-2n,b-4))\\
\chi(\Omega^1_P|\cX(a,b))&=&\chi(\Omega^1_P(a,b))-2\chi(\Omega^1_P(a-n,b-2))+\chi(\Omega^1_P(a-2n,b-4)).
\end{array}
$$
The formulas
$$ 
\chi(\cO_P(a,b))=(a+1)\binom{b+4}{4},
$$
$$
\chi(\Omega^1_P(a,b))=(a-1)\binom{b+4}{4}+\left(5\binom{b+3}{4}-\binom{b+4}{4}\right)(a+1),
$$
thus give
$$\chi(\Omega^1_{\cX})=-7+20n, \quad 
\chi(\cX)=16-40n.$$
When $n>0$ we obtain 
$$h^2(\Omega^1_{\cX})=-5+20n$$
as well.

\begin{rema}\label{rema:heightzero}
The special case $n=0$ is a constant family of quartic del Pezzo surfaces,
which clearly has sections.
\end{rema}

\subsubsection*{Odd part}
Now assume $d=2n+1$ so (\ref{eq:bundle}) takes the form
$$
\pi_* I_{\cX}(2) \simeq \cO_{\bP^1}(-n) \oplus \cO_{\bP^1}(-n-1)
$$
and $\cX$ is a complete intersection of forms of bidegrees $(n,2)$ and $(n+1,2)$.  
In this situation
$$\omega^{-1}_{\pi}=\cO_{\cX}(-2n-1,1), \quad \pi_*\omega^{-1}_{\pi}=\cO_{\bP^1}(-2n-1)\otimes \Gamma(\cO_{\bP^4}(1))$$
whence 
$$\deg(\pi_*\omega^{-1}_{\pi})=-10n-5, \quad h(\cX)=20n+10.$$

\begin{rema} \label{rema:heightten}
The base case $n=0$ is studied in Section~\ref{sect:test}, where we show that
it always admits a section $\sigma:\bP^1 \ra \cX$, defined over the coefficient field.  
\end{rema}

\subsection*{Case 2}
We next consider 
$$\cX \subset P:=\bP^1 \times \bP^5$$
contained in a complete intersection of a form $L$ of bidegree $(1,1)$.

\subsubsection*{Even part}
Consider when $\cX$ 
is a complete intersection of a form $L$ of bidegree $(1,1)$
and two forms of bidegree $(n,2)$.
Let $V$ be defined by the exact sequence
$$0 \ra \cO_{\bP^1}(-1) \stackrel{\cdot L}{\ra} 
\Gamma(\cO_{\bP^5}(1))\otimes \cO_{\bP^1} \ra V \ra 0.$$
For generic $L$, we have
$$V \simeq \cO_{\bP^1}^4 \oplus \cO_{\bP^1}(1).$$

We have
$$\omega_{\pi}^{-1}=\cO_{\cX}(-2n-1,1)$$
and the projection formula yields
$$\pi_*\omega_{\pi}^{-1}=\cO_{\bP^1}(-2n-1)\otimes V,$$
whence
$$\deg(\pi_*\omega_{\pi}^{-1})=4(-2n-1)-2n=-10n-4, \quad
h(\cX)=20n+8.$$

\begin{rema} \label{rema:heighteight}
We analyze rational points in the case $n=0$.   Consider the projection
$\cY:=\pi_2(\cX)\subset \bP^5$ onto the second factor;  it is a complete intersection
of two quadrics.  The fibration corresponds to a pencil of hyperplane sections of $\cY$.
Let $Z \subset \cY$ be its base locus, so that $\cX=\Bl_Z(\cY)$.  Since $\cX$ is
smooth, $Z$ must be a smooth curve of genus one.

The constant subfamily
$$\begin{array}{ccc}
Z \times \bP^1  & \subset & \cX \\
\downarrow  & \swarrow &   \\
\bP^1   &  &
\end{array}
$$
has a section over the coefficient field $k$, provided $Z(k) \neq \emptyset$.  
\end{rema}

\subsubsection*{Odd part}
Now consider 
$$\cX \subset P:=\bP^1 \times \bP^5$$
as a complete intersection of a form $L$ of bidegree $(1,1)$,
a form of bidegree $(n,2)$, and a form of bidegree $(n+1,2)$.  
We have
$$\omega^{-1}_{\pi}=\cO_{\cX}(-2n-2,1)$$
so that
$$\pi_*\omega_{\pi}^{-1}=\cO_{\bP^1}(-2n-2) \otimes V,$$
where $V$ is as above.  
Thus we conclude
$$\deg(\pi_*\omega_{\pi}^{-1})=4(-2n-2)+(-2n-1)=-10n-9,
\quad h(\cX)=20n+18.$$

\subsection*{Case 3}

Consider
$$\cX \subset \bP^1 \times \bP^6$$
as a complete intersection of two forms $L_1$ and $L_2$ of bidegree $(1,1)$ 
and two additional
forms quadratic in the second factor.  
Let $V$ be defined by the exact sequence
$$0 \ra \cO_{\bP^1}(-1)^2 \stackrel{\cdot (L_1,L_2)}{\ra} 
\Gamma(\cO_{\bP^6}(1))\otimes \cO_{\bP^1} \ra V \ra 0;$$
for generic $L_1$ and $L_2$, we have
$$V\simeq \cO_{\bP^1}(1)^2 \oplus \cO_{\bP^1}^3.$$
As before, we have 
$$\pi_*\cO_{\cX}(1)\simeq V.$$

\subsubsection*{Even part}
Suppose that the additional quadratic forms both have degree $(n,2)$.
Then $\omega_{\pi}\simeq \cO_{\cX}(2n+2,-1)$  
whence
$$\pi_*\omega_{\pi}^{-1}
\simeq \cO_{\bP^1}(-2n-1)^2 \oplus \cO_{\bP^1}(-2n-2)^3$$
and 
$$\deg(\pi_*\omega_{\pi}^{-1})=-10n-8, \quad
h(\cX)=20n+16.$$

\subsubsection*{Odd part}
Suppose that the additional forms have degrees $(n,2)$ and
$(n+1,2)$. 
Then $\omega_{\pi}\simeq \cO_{\cX}(2n+3,-1)$  
whence
$$\pi_*\omega_{\pi}^{-1}
\simeq \cO_{\bP^1}(-2n-2)^2 \oplus \cO_{\bP^1}(-2n-3)^3$$
and 
$$\deg(\pi_*\omega_{\pi}^{-1})=-10n-13, \quad
h(\cX)=20n+26.$$

\begin{exam}[$n=-1$ case] \label{exam:heightsix}
We do the $n=-1$ case via a slightly different construction:  Consider
the intersection of the forms $L_1$ and $L_2$:
$$\bP(V^{\vee})=\bP(\cO_{\bP^1}(-1)^2 \oplus \cO_{\bP^1}^3) \subset \bP(\cO_{\bP^1}^7).$$
Note that 
$$\Gamma(\cO_{\bP(V^{\vee})}(a)\otimes \pi^*\cO_{\bP^1}(b))=\Gamma(\Sym^a(V)\otimes \cO_{\bP^1}(b)),$$
which is nonzero if and only if $a+b \ge 0$, e.g., if $a=2$ and $b=-1$.  
In this case, we have 
$$\Sym^2(V) \otimes \cO_{\bP^1}(-1)=\cO_{\bP^1}(1)^3\oplus \cO_{\bP^1}^6 \oplus
			\cO_{\bP^1}(-1)^6,$$
i.e., the space of sections has dimension $12$.  Furthermore, these all vanish along 
the distinguished two-dimensional projective subbundle of $\bP(V^{\vee})$.  
Since non-singular quadric threefolds contain no two-planes,
divisors $D$ in $\cO_{\bP(V^{\vee})(2)}\otimes \pi^*\cO_{\bP^1}(-1)$ are fibered over $\bP^1$
in singular quadric hypersurfaces.

Let $\cX=D\cap Q$ where $Q$ has bidegree $(0,2)$, i.e., $Q$ is the pullback of a quadric from $\bP^6$.  
The image of $\bP(V^{\vee})$ in $\bP^6$ is 
a quadric hypersurface. 

Since $\cX$ is contained in a {\em singular} quadric hypersurface, its mono\-dromy is non-maximal;  indeed, $W(D_5)$
acts on the five singular quadric hypersurfaces containing $\cX$ 
via the canonical homomorphism $W(D_5)\ra \fS_5$.   This is consistent
with Proposition~\ref{prop:monodromy}.  
\end{exam}

\begin{rema} \label{rema:heightsix}
Let $C$ denote the intersection of the distinguished two-plane in $\bP^4$ with $Q$;  we have a constant subfamily
$$C\times \bP^1 \subset \cX.$$
It follows that $\cX \ra \bP^1$ admits sections over any ground field.  
\end{rema}

\subsection*{Case 4}

Consider
$$\cX \subset \bP^1 \times \bP^7$$
as a complete intersection of three forms $L_1, L_2,$ and $L_3$ of bidegree $(1,1)$ 
and two additional
forms quadratic in the second factor.  
Let $V$ be defined by the exact sequence
$$0 \ra \cO_{\bP^1}(-1)^3 \stackrel{\cdot (L_1,L_2,L_3)}{\ra} 
\Gamma(\cO_{\bP^7}(1))\otimes \cO_{\bP^1} \ra V \ra 0;$$
for generic $L_i$, we have
$$V\simeq \cO_{\bP^1}(1)^3 \oplus \cO_{\bP^1}^2.$$

\subsubsection*{Even part}
Suppose that the additional quadratic forms both have degree $(n,2)$.
Then $\omega_{\pi}\simeq \cO_{\cX}(2n+3,-1)$  
whence
$$\pi_*\omega_{\pi}^{-1}
\simeq \cO_{\bP^1}(-2n-2)^3 \oplus \cO_{\bP^1}(-2n-3)^2$$
and 
$$\deg(\pi_*\omega_{\pi}^{-1})=-10n-12, \quad
h(\cX)=20n+24.$$

\begin{exam}[$n=-1$ case] \label{exam:heightfour}
Consider the intersection of the forms $L_1,L_2,$ and $L_3$:
$$\bP(V^{\vee})=\bP(\cO_{\bP^1}(-1)^3 \oplus \cO_{\bP^1}^2) \subset \bP(\cO_{\bP^1}^8).$$
In this case, we have 
$$\Sym^2(V) \otimes \cO_{\bP^1}(-1)=\cO_{\bP^1}(1)^6\oplus \cO_{\bP^1}^6 \oplus
			\cO_{\bP^1}(-1)^3,$$
i.e., the space of sections has dimension $18$.  Furthermore, these all vanish along 
the projectivization of the distinguished one-dimensional projective subbundle of $\bP(V^{\vee})$.  
Divisors $D$ in $\cO_{\bP(V^{\vee})(2)}\otimes \pi^*\cO_{\bP^1}(-1)$ are fibered over $\bP^1$
in quadrics containing a distinguished line $\ell$.  
Let $\cX=D_1\cap D_2$ for two divisors in this linear series.  
\end{exam}

\begin{rema} \label{rema:heightfour}
Since $D_1$ and $D_2$ {\em both} contain $\ell$, we have a constant subfamily:
$$\begin{array}{ccc}
\ell \times \bP^1  & \subset & \cX \\
\downarrow  & \swarrow &   \\
\bP^1   &  &
\end{array}
$$
This provides sections for $\cX \ra \bP^1$.  Note here that $\cX$ is non-minimal,
as $\ell$ can be blown down over the ground field.  
\end{rema}

\subsubsection*{Odd part}
Suppose that the additional forms have degrees $(n,2)$ and
$(n+1,2)$. 
Then $\omega_{\pi}\simeq \cO_{\cX}(2n+4,-1)$  
whence
$$\pi_*\omega_{\pi}^{-1}
\simeq \cO_{\bP^1}(-2n-3)^3 \oplus \cO_{\bP^1}(-2n-4)^2$$
and 
$$\deg(\pi_*\omega_{\pi}^{-1})=-10n-17,\quad
h(\cX)=20n+34.$$

\begin{exam}[$n=-1$ case] \label{exam:heightfourteen}
Retain the notation of Example~\ref{exam:heightfour}.
Here we let $\cX=D\cap Q$ where $Q$ has bidegree $(0,2)$.  
The image of $\bP(V^{\vee})$ in $\bP^7$ lies on 
three quadrics, and is in fact a cone over a Segre threefold.  

The intersection of $\ell$ with $Q$ gives two distinguished sections
$$\sigma_1,\sigma_2:\bP^1 \ra \cX,$$
which may be conjugate over the ground field.  
We shall study this case in Section~\ref{sect:heightlarger}.
\end{exam}

\subsection*{Case 5}

Consider
$$\cX \subset \bP^1 \times \bP^8$$
as a complete intersection of four forms $L_1, L_2, L_3$ and $L_4$ of bidegree $(1,1)$ 
and two additional
forms quadratic in the second factor.  
Let $V$ be defined by the exact sequence
$$0 \ra \cO_{\bP^1}(-1)^4 \stackrel{\cdot (L_1,L_2,L_3,L_4)}{\ra} 
\Gamma(\cO_{\bP^8}(1))\otimes \cO_{\bP^1} \ra V \ra 0;$$
for generic $L_i$, we have
$$V\simeq \cO_{\bP^1}(1)^4 \oplus \cO_{\bP^1}.$$

\subsubsection*{Even part}
Suppose that the additional quadratic forms both have degree $(n,2)$.
Then $\omega_{\pi}\simeq \cO_{\cX}(2n+4,-1)$  
whence
$$\pi_*\omega_{\pi}^{-1}
\simeq \cO_{\bP^1}(-2n-3)^4 \oplus \cO_{\bP^1}(-2n-4)$$
and 
$$\deg(\pi_*\omega_{\pi}^{-1})=-10n-16,
\quad h(\cX)=20n+32.$$

\begin{exam}[$n=-1$ case] \label{exam:heighttwelve}
Consider
the intersection of the forms $L_1,L_2,L_3$ and $L_4$:
$$\bP(V^{\vee})=\bP(\cO_{\bP^1}(-1)^4 \oplus \cO_{\bP^1}) \subset \bP(\cO_{\bP^1}^9).$$
In this case, we have
$$\Sym^2(V) \otimes \cO_{\bP^1}(-1)=\cO_{\bP^1}(1)^{10}\oplus \cO_{\bP^1}^4 \oplus
                        \cO_{\bP^1}(-1),$$
and the space of sections has dimension $24$.  These all vanish along a distinguished section $\sigma:\bP^1 \ra \bP(V^{\vee})$,
corresponding the projectivization of the trivial summand of $V^{\vee}$;  let $D_1$ and $D_2$ denote two
of the divisors in this class.  
Let $\cX=D_1\cap D_2$, which is a quartic del Pezzo fibration.
\end{exam}

\begin{rema}
\label{rema:heighttwelve}
The section $\sigma$ factors through $\cX$.  This canonical section 
is defined over any field of scalars.  
\end{rema}

\subsubsection*{Odd part}
Suppose that the additional forms have degrees $(n,2)$ and
$(n+1,2)$. 
Then $\omega_{\pi}\simeq \cO_{\cX}(2n+5,-1)$  
whence
$$\pi_*\omega_{\pi}^{-1}
\simeq \cO_{\bP^1}(-2n-4)^4 \oplus \cO_{\bP^1}(-2n-5)$$
and 
$$\deg(\pi_*\omega_{\pi}^{-1})=-10n-21, \quad
h(\cX)=20n+42.$$
We can realize the $n=-1$ case via constructions similar to Examples~\ref{exam:heightfour}, \ref{exam:heightsix}, \ref{exam:heightfourteen}, and
\ref{exam:heighttwelve}.  

\begin{exam} \label{exam:heighttwo}
We do not have examples of height-two quartic del Pezzo fibrations $\cX \ra \bP^1$ with smooth total space
and fibers having at worst one ordinary double point.

A direct generalization of the approach in Examples~\ref{exam:heightfour}, \ref{exam:heightsix}, \ref{exam:heightfourteen}, and
\ref{exam:heighttwelve} yields a pencil of {\em singular} del Pezzo surfaces.  The bundle corresponding 
to the bidegree $(-2,2)$ forms is
$$\Sym^2(V) \otimes \cO_{\bP^1}(-2)=\cO_{\bP^1}^{10}\oplus \cO_{\bP^1}(-1)^4 \oplus
                        \cO_{\bP^1}(-2),$$
parametrizing quadric hypersurfaces singular along the distinguished section of $\bP(V^{\vee})$.  The corresponding
divisors $D$ depend on $9$ parameters.  The bundle corresponding to bidegree $(-1,2)$ forms is
$$\Sym^2(V) \otimes \cO_{\bP^1}(-1)=\cO_{\bP^1}(1)^{10}\oplus \cO_{\bP^1}^4 \oplus
                        \cO_{\bP^1}(-1),$$
parametrizing quadric hypersurfaces containing the distinguished section.  The
associated divisor $D'$ depends on $24-1-2=21$ parameters.  Altogether, this construction 
depends on
$$56+9+21-(3+80)=3$$
parameters, whereas the expected dimension is two.  This reflects the higher cohomology of the 
bundle $\Sym^2(V)(-2)$.  
\end{exam}

\section{Structure of height two families} 
\label{sect:structwo}

We use the notation of Example~\ref{exam:heighttwo}.  Projecting from
the distinguished section $p:\bP^1 \ra \cX$ gives a rational map
$$\bP(\cO_{\bP^1}(-1)^4 \oplus \cO_{\bP^1}) \dashrightarrow 
\bP(\cO_{\bP^1}(-1)^4)\simeq \bP^3 \times \bP^1.$$
The image of $D$ under this mapping is a smooth, trivial quadric surface
fibration
$$\cQ:=Q\times \bP^1 \subset \bP^3 \times \bP^1\stackrel{\varpi}{\ra}\bP^1.$$
Let $\cE=E\times \bP^1$ denote the exceptional divisor over $p(\bP^1)$;  $E\subset Q$ is
a hyperplane section.

The equation of $D'$ takes the form
$$H(u_1,u_2,u_3,u_4)u_0+sF(u_1,u_2,u_3,u_4)+tG(u_1,u_2,u_3,u_4)=0,$$
where $H$ is linear and $F$ and $G$ are quadratic.  Here
$s$ and $t$ are homogeneous coordinates on $\bP^1$, $u_1,\ldots,u_4$
are homogeneous coordinates on $\bP^3$, and $[u_0,u_1,\ldots,u_4]=[1,0,\ldots,0]$
is the distinguished section.  Thus $\cX \ra \cQ$ blows up the curve
$$Z=\{H=sF+tG=0 \} \subset \cE=E\times \bP^1=\{H=0\},$$
a type $(4,1)$ curve in $E \times \bP^1\simeq \bP^1 \times \bP^1$.  
The projection $\varpi_2:\cE \ra E$ induces $Z\simeq E$.  
In summary, we have
$$\begin{array}{rcccl}
	&	& \widetilde{\cX}=\Bl_{p(\bP^1)}(\cX)=\Bl_{Z}(\cQ) &  & \\
        &\swarrow&				   & \searrow & \\
\cX     &	 &				   &	      & \cQ   
\end{array}
$$
and
$$h^1(\Omega^2_{\widetilde{\cX}})=h^2(\Omega^1_{\widetilde{\cX}})=0.$$

\

We enumerate sections $\sigma:\bP^1 \ra \cX$ disjoint from the singular section.

\

\noindent
{\bf height $0$:} \
The points of 
$$Q\setminus (Q\cap \{H=0\})\simeq (\bP^1 \times \bP^1)\setminus\{\mathrm{diagonal}\}$$
yield height-zero sections of $\pi:\cX \ra \bP^1$ disjoint from $p(\bP^1)$.

\

\noindent
{\bf height $1$:} \
Fix a line $L \subset Q$ from one of the two rulings, which meets $E$ in a single point.
The subfamily $L \times \bP^1 \subset \cQ$ meets $Z$ in a single point $z$.  Curves of bidegree
$(1,1)$ in $L\times \bP^1$ containing $z$ form a two-parameter family.  The proper
transforms of these curves in $\cX$ are sections disjoint from $p(\bP^1)$.  Altogether,
we obtain a three parameter family of sections, birational to $\bP^1 \times \bP^2$.
Taking the choice of ruling into account, we obtain two such families.

Given a generic point $x' \in \cX$, we exhibit height-one sections
of $\pi$ containing $x'$.   Regard $x' \in Q \times p', p'= \pi(x')$ and choose the ruling
so it contains the projection of $x'$ onto $Q$.  
These are indexed by {\em two} copies of $\bP^1\setminus \{\text{two points}\}$,
one for each ruling of $Q$.

\

\noindent
{\bf height $2$:} \  Fix a hyperplane section $M\subset Q$, meeting $E$ in two points;
$M\times \bP^1$ meets $Z$ in $\{z_1,z_2\}$.  Curves of bidegree $(1,1)$ in $M\times \bP^1$
containing $\{z_1,z_2\}$ form a one-parameter family.  This construction gives a 
four parameter family of sections of $\cX\ra \bP^1$, birational to $\bP^3 \times \bP^1$.

\

\noindent
{\bf height $3$:} \  Repeat the same argument for curves of bidegree $(1,2)$ (twisted
cubics) in $Q$.

\

\noindent
{\bf height $4$:} \  Repeat the same argument for curves of bidegree $(1,3)$ in $Q$.

Consider nodal rational curves $M' \subset Q$ of bidegree $(2,2)$,
which meet $E$ in four points.  Let $M$ denote the normalization of $M'$;  $M\times \bP^1$ meets
$Z$ in four points, $\{z_1,z_2,z_3,z_4\}$.  Consider curves of bidegree $(1,1)$ in $M\times \bP^1$
containing these points.  Such curves exist only when we have a projective equivalence
$$(M;z_1,\ldots,z_4)\sim (\bP^1,\varpi(z_1),\varpi(z_2),\varpi(z_3),\varpi(z_4)).$$ 

This leads to the following geometric problem:
\begin{quote}
Let $Q$ be a smooth quadric surface, $E\subset Q$ a smooth hyperplane, $z_1,\ldots,z_4 \in E$
prescribed points, and $w_1,\ldots,w_4 \in \bP^1$ prescribed points.
Describe the morphisms
$$j:\bP^1 \ra Q$$
of bidegree $(2,2)$ such that $j(w_i)=z_i$ for $i=1,\ldots,4$.  
\end{quote}
These are constructed as follows:  For any point $q\in Q$, there exists a curve of bidegree $(2,2)$ in $Q$ with a node 
at $q$, containing $z_1,\ldots,z_4$, and satisfying the condition on projective equivalence:
Indeed, projection from $q$ converts the problem to finding a conic through four given points
in $\bP^2$ so that the points have prescribed projective equivalence class in the conic.

\

\noindent
{\bf height $5$:}  The case of bidegree $(1,4)$ proceeds just as the case of twisted
cubics.  

The case of bidegree $(2,3)$ boils down to the following:
\begin{quote}
Let $Q$ be a smooth quadric surface, $E\subset Q$ a smooth hyperplane, $z_1,\ldots,z_5 \in E$
prescribed points, and $w_1,\ldots,w_5 \in \bP^1$ prescribed points.
Describe the morphisms
$$j:\bP^1 \ra Q$$
of bidegree $(2,3)$ such that $j(w_i)=z_i$ for $i=1,\ldots,5$.  
\end{quote}

These are constructed as follows:  Fix pairs of pointed stable curves
$$(E;z_1,\ldots,z_5), \quad (\bP^1;w_1,\ldots,w_5), \quad E\simeq \bP^1$$
and let $C$ denote the stable curve of genus four obtained by gluing
$z_i$ to $w_i$ for $i=1,\ldots,5$.  The canonical embedding of $C\subset \bP^3$
is contained in a quadric surface  
$$Q'\simeq \bP^1 \times \bP^1 \subset \bP^3;$$
we may assume the components
$E$ and $\bP^1$ have types $(1,2)$ and $(2,1)$ respectively.  
For any two general points
$e_1,e_2 \in E$, consider the linear series of curves of type
$(1,2)$ with assigned basepoints at $e_1$ and $e_2$.  
This yields a morphism to a quadric surface
$$\Bl_{e_1,e_2}Q' \ra Q \subset \bP^3.$$
The image of $E$ is a hyperplane class;  the image of $\bP^1$ is of type 
$(2,3)$.  

Thus for each collection $\{z_1,\ldots,z_5 \} \subset E$, 
we obtain a two-parameter family of sections $\sigma:\bP^1 \ra \cX$,
corresponding to $M'\subset Q$ of bidegree $(2,3)$ 
such that
\begin{itemize}
\item{$M' \cap Q = \{z_1,\ldots,z_5 \}$;}
\item{$(M;z_1,\ldots,z_5)\simeq (\bP^1; \pi(z_1),\ldots,\pi(z_5))$.}
\end{itemize}
The two parameters reflect the choice of $(e_1,e_2)$ above.
Hence we get a seven-dimensional rational variety parametrizing these
sections.

\

\noindent
{\bf height $6$:}  
We focus on the case of bidegree $(3,3)$, which leads to the following
statment:
\begin{quote}
Let $Q$ be a smooth quadric surface, $E\subset Q$ a smooth hyperplane, $z_1,\ldots,z_6 \in E$
prescribed points, and $\phi:E\simeq \bP^1 \ra \bP^1$ a fixed mapping of degree four.
The morphisms 
$$j:\bP^1 \ra Q$$
of bidegree $(3,3)$ such that 
$$(\bP^1;\phi(z_1),\ldots,\phi(z_6))\sim (E;z_1,\ldots,z_6)$$
are parametrized by a rational variety.  
\end{quote}
The rational curves of bidegree $(3,3)$ through six fixed points depend on $15-4-6=5$
parameters.

This is proved as follows:  Let $C$ be a smooth projective curve of genus five that is 
Brill-Noether general, i.e., not hyperelliptic, trigonal, or a plane quintic.  Then the
canonical embedding
$$C\hookrightarrow \bP^4$$
realizes $C$ as a complete intersection of three quadrics.  In particular, $C$ admits
embeddings into a family of quadric del Pezzo surfaces parametrized by 
$\bP(I_C(2)^{\vee})\simeq \bP^2$.  This remains true for certain stable curves, e.g.,
$$C=C_1 \cup_{z_1,\ldots,z_6} C_2, \quad C_1\simeq C_2 \simeq \bP^1,$$
where $z_1,\ldots,z_6$ are chosen arbitrarily on each component. 

Fixing $(C_1;z_1,\ldots,z_6)$ and $(C_2;z_1,\ldots,z_6)$, let $S$ denote one of the 
del Pezzo surfaces arising from this construction.  The linear series
$|C_1|$ induces a birational morphism onto a quadric
$$S\ra Q.$$
The image of $C_1$ is $Q$ is a hyperplane section;  the image of $C_2$ is a bidegree
$(3,3)$ curve with four nodes.

\section{Structure of height four families} 
\label{sect:strucfour}

We give additional details about the geometric structure of height 
four families.  Recall the notation of Example~\ref{exam:heightfour}.  
The generic fiber $X$ contains a line $\ell$.  Projecting from
this line gives a birational morphism
$$X \ra \bP^2$$
blowing down the five lines incident to $\ell$;  $\ell$
is mapped to the plane conic containing the images of 
the five lines.

We interpret this for $\cX$ over $\bP^1$:  Fiberwise projection
from $\ell \times \bP^1$ induces a rational map
$$\bP(V^{\vee})=\bP(\cO_{\bP^1}(-1)^3 \oplus \cO_{\bP^1}^2)
 \dashrightarrow \bP(\cO_{\bP^1}(-1)^3)\simeq \bP^2 \times \bP^1,$$
inducing a birational morphism
$$\beta:\cX \ra \bP^2 \times \bP^1.$$
The image $L:=\beta(\ell \times \bP^1)\subset \bP^2 \times \bP^1$ is a
product $\{\mathrm{conic}\} \times \bP^1$.  
Let $\varpi_j,j=1,2$ denote the projections from $\bP^2 \times \bP^1$.  

We describe the center $Z$ of $\beta$:  It is a rational curve of bidegree $(5,1)$
in $\ell \times \bP^1$, i.e., projection onto the second factor has degree
five and eight branch points.  It follows that
$$h^2(\Omega^1_{\cX})=0, \quad h^1(\Omega^1_{\cX})=3.$$
The composition
$$\bP^2 \times \bP^1 \stackrel{\beta^{-1}}{\dasharrow} \cX \stackrel{\pi_2}{\ra} \cY$$
is given by forms of bidegree $(3,1)$ on $\bP^2 \times \bP^1$ vanishing 
along $Z$.  Its restriction to $\ell\times \bP^1 \subset \cX$ has bidegree
$(2\cdot 3,1)-(5,1)=(1,1)$, thus collapses
$\ell \times \bP^1$ onto the first factor.

We enumerate sections of $\pi:\cX \ra \bP^1$ of small height $h$, restricting our
attention to sections that dominate $\cX$.  For example, we do not consider
sections contained in $\ell \times \bP^1$.  Furthermore, we subdivide
based on how the sections meet $\ell \times \bP^1$ and $\cE$, the exceptional
divisor of $\beta$.  

\

\noindent
{\bf height $0$:} \
The constant sections of $\bP^2 \times \bP^1 \ra \bP^1$ induce
a {\em two-parameter} family of sections for $\pi:\cX \ra \bP^1$, disjoint
from $\ell\times \bP^1$ and $\cE$.

\

\noindent
{\bf height $1$:} \
Consider sections of $\varpi_2:\bP^2 \times \bP^1 \ra \bP^1$ 
projecting onto a line $L\subset \bP^2$, meeting $Z$ in $\{z_1,z_2\}$ with
$\varpi_1(z_1,z_2)=L \cap \ell$.  These depend on three parameters, are 
disjoint from $\ell \times \bP^1$, but meet $\cE$ twice.  These sections
are parametrized by a space fibered over $\Sym^2(Z)\simeq \bP^2$, with
generic fiber $\bP^1$.

\

\noindent
{\bf height $2$:} \
We have sections of $\varpi_2$
projecting onto a line $L$, meeting $Z$ at one point.  The resulting 
sections of $\cX \ra \bP^1$ depend on four parameters, and meet both $\cE$
and $\ell \times \bP^1$ once.

There are also sections with $\pi_1(\sigma(\bP^1))=M$ a plane conic, where the
corresponding curve in $M\times \bP^1$ meets $Z$ in four points $\{z_1,z_2,z_3,z_4\}$
with $\pi_1(z_1,z_2,z_3,z_4)=M\times \ell$.  These are birational to $\Sym^4(Z)$;
indeed, there exists a unique plane conic $M\supset \varpi_1(z_1),\ldots,\varpi_1(z_4)$
such that the projective invariant of these four points {\em in $M$} equals the 
projective invariant of the points $\pi(z_1),\ldots,\pi(z_4)$ in $\bP^1$.

\

\noindent
{\bf height $3$:} \
Sections of $\varpi_2$ projecting onto a line and disjoint from $Z$.
These are parametrized by an open subset $(\bP^2)^{\vee} \times \bP^3$;
the generic fiber over a point of $(\bP^2)^{\vee}$ is isomorphic to 
$\bP^3 \setminus \{\text{two hyperplanes} \}$.  The resulting sections of 
$\pi$ are disjoint from $\cE$ and meet $\ell \times \bP^1$ twice.

There are sections of $\varpi_2$ projecting onto a conic, meeting $Z$ in three points.  We have a 
unique such section over each conic in $\bP^2$, so the space of these sections is 
birational to $\bP^5$.  The corresponding sections of $\pi$ meet $\cE$ three times and 
$\ell \times \bP^1$ one.  

Sections of $\varpi_2$ projecting onto a nodal cubic curve $M'$ with normalization $M$;
the corresponding curve in $M\times \bP^1$ meets $Z$ in six points $\{z_1,\ldots,z_6\}$.  
We have the constraint that
$$(M;z_1,\ldots,z_6) \sim (\bP^1;\pi(z_1),\ldots,\pi(z_6)),$$
i.e., these are projectively equivalent.  Projection from the node of $M'$ induces a degree-two cover
$$\ell \ra M$$
and $\ell\simeq Z$ via $\varpi_1$.  This translates into the following problem
\begin{quote}
Fix a degree five morphism $\phi:\bP^1 \ra \bP^1$.  Characterize sets of points 
$$(z_1,\ldots,z_6) \in (\bP^1)^6$$
such that there exists some degree-two morphism $\psi:\bP^1 \ra \bP^1$ with 
$$(\bP^1;\phi(z_1),\ldots,\phi(z_6))=(\bP^1;\psi(z_1),\ldots,\psi(z_6)).$$
\end{quote}
Note that $\psi$ coincides with projection from the node of $M'$ and determines
this node uniquely.  

This example can be approached using the Brill-Noether analysis employed in
Section~\ref{sect:structwo}.  Consider pairs of $6$-pointed stable curves
of genus zero
$$(C_1;z_1,z_2,\ldots,z_6), \quad
(C_2;z_1,z_2,\ldots,z_6), 
$$
and let $C$ denote the genus-five stable curve obtained by gluing these
together.  If $C$ is immersed in $\bP^2$ as the union of a conic and a nodal cubic
$$M' \cup \ell \subset \bP^2$$
then the canonical curve $j:C \hookrightarrow \bP^4$ lies
on a cubic ruled surface, i.e., the blow up of $\bP^2$ at the node.  This is
equivalent to $C$ being trigonal, with the $g^1_3$ given by the rulings.  
Trigonal curves are codimension one in the moduli space of curves of
genus five.

The projection of the trigonal curves $C$ as above to
$M_{0,6} \times M_{0,6}$ is finite onto its image, by a direct
computation with MAPLE.  This image is unirational as it is
dominated by $\bP^5 \times \bP^6$---the choice of a plane conic
and a plane cubic with a node at a prescribed point.

\section{Structure of height six families}
\label{sect:prepretest}
Recall the notation leading into Remark~\ref{rema:heightsix}.  In 
particular, we have
$$\bP(V^{\vee})=\bP(\cO_{\bP^1}(-1)^2 \oplus \cO_{\bP^1}^3)
\stackrel{\pi_2}{\ra} \bP^6,$$
where the image is quadric hypersurface.  It contains a distinguished
plane $\Pi=\pi_2(\bP(\cO_{\bP^1}^3))$.
The divisor
$D\subset \bP(V^{\vee})$ has image $\pi_2(D)$ isomorphic to
a cone over the Segre threefold $\bP^1 \times \bP^2 \subset \bP^5$.
The plane $\Pi$ is the cone over $\bP^1 \times \{p\}$
for a point $p \in \bP^2$.  

Since $\cY$ is the intersection of $\pi_2(D)$ with a quadric hypersurface $Q$,
we may interpret it as a double cover of $\bP^1 \times \bP^2$ branched over
a divisor of bidegree $(2,2)$.  In particular, $\cY$ is a quadric surface
bundle over $\bP^1$ with six degenerate fibers, and a conic bundle
over $\bP^2$, with discriminant a plane quartic $Z\subset \bP^2$
and associated \'etale double cover $\gamma:\tilde{Z}\ra Z$.
Thus the intermediate Jacobian 
$$\IJ(\cY)=\Prym(\tilde{Z}\ra Z)=\JJ(C),$$
where $C$ is the genus two curve branched over the discriminant
of $\cY \ra \bP^1$.  
Thus in particular, we have
$$h^2(\Omega^1_{\cX})=2,  \
h^1(\Omega^1_{\cX})=3,$$
 which reflects the fact that the fibration
$\pi:\cX \ra \bP^1$ has relative Picard rank two.

The morphism $\pi_2:\cX \ra \cY$ is the blow up of $\cY$ along
the curve $T=\Pi \cap Q$, which coincides with the fiber of
$\cY \ra \bP^2$ over $p$.  The fibration 
$$\pi=\pi_1:\cX \ra \bP^1$$
is induced by the projection from $p$
$$\varpi_p:\bP^2 \dashrightarrow \bP^1.$$

We construct a family of sections of $\pi:\cX \ra \bP^1$ mapping
birationally onto $\IJ(\cX)\simeq \Prym(\tilde{Z}\ra Z).$ 
Let $\check{\bP}^2$ denote the projective space dual to the $\bP^2$
arising as a factor of the Segre threefold.  For each 
$L \in \check{\bP}^2$, we have a conic bundle
$$\phi_L:\cY\times_{\bP^2}L \ra L,$$
degenerate over $L\cap Z$.  This conic bundle has eight
$(-1)$-curves $E_1,\ldots,E_8$ that induce
sections $\{s_1,\ldots,s_8\}$ of $\phi_L$;  
we obtain a generically finite morphism
$$\cE \ra \check{\bP}^2,$$
where $\cE$ is a smooth projective surface compactifying the
space parametrizing $E_1,\ldots,E_8$ as $L$ varies.  In particular,
there is an open set $\cE^{\circ} \subset \cE$ and a correspondence
$$\begin{array}{rcl}
\cS & \ra & \cY \\
\downarrow &  & \\
\cE^{\circ}
\end{array}$$
where the vertical arrow is a $\bP^1$-bundle and the horizontal
arrow is induced by the inclusions 
$E_1\cup \ldots \cup E_8 \subset \cY\times_{\bP^2}L\subset \cY$.  

Now each $E_j$ meets each of the four degenerate fibers of $\phi_L$
in one component;  these components are indexed by $z_1,z_2,z_3,z_4\in \tilde{Z}$. 
Note that $\gamma(z_1,z_2,z_3,z_4)=L\cap Z \in |K_Z|$.  
Thus we obtain maps
$$\cE \dashrightarrow \Pic^4(\tilde{Z}) \stackrel{\gamma_*}{\ra} \Pic^4(Z),$$
where the image of $\cE$ lies in the fiber over $K_Z$.  Thus there is an induced
mapping 
$$\cE \dashrightarrow \Prym(\tilde{Z} \ra Z),$$
which is clearly birational.  

We claim that the fibers of $\cS \ra \cE$ yield sections
of $\pi:\cX \ra \bP^1$, i.e., $\cE$ parametrizes a family of sections
birational to $\IJ(\cX)$.  Indeed, fix $[E_j]\in \cE$ over $L\subset \bP^2$
and a fiber $\cX_r=\pi^{-1}(r)$ for $r\in \bP^1$.  
The point $r$ corresponds to a line 
$$\ell_r:=\varpi_p^{-1}(r) \subset \bP^2$$
passing through $p$.
There is a unique intersection point $p'=L \cap \ell_r$,
and the fiber $\cY\times_{\bP^2}\{p'\}$ meets $E_j$ in
precisely one point.

\section{Structure of height eight families}
\label{sect:pretest}
We return to the example discussed in Remark~\ref{rema:heighteight}:   
Let $\cY \subset \bP^5$ denote a complete intersection of two quadrics,
$E\subset \cY$ a smooth codimension-two linear section, and 
$\cX=\Bl_E(\cY)$.  The pencil of linear sections containing $E$ induced
a del Pezzo fibration
$$\pi:\cX \ra \bP^1.$$

We review the classical constructions for the lines on $\cY$, following 
Miles Reid's thesis \cite{Reid} (see also \cite{Tjurin} and \cite{Castravet} for
higher-degree curves).  
The discriminant divisor restricted to the pencil of quadrics 
yields six points on $\bP^1$, thus determining a genus two 
curve $C$.  Note that the space of maximal isotropic subspaces on
quadrics in the pencil is a $\bP^3$-bundle over $C$;  these
correspond to the conics on $\cY$.
Projecting from a line $\ell \subset \cY$,
we obtain a morphism
$$\Bl_{\ell}(\cY) \ra \bP^3,$$
blowing up $C$, realized as a quintic curve in $\bP^3$.  
The space of all lines 
$$F_1(\cY)\simeq \JJ(C),$$
which in turn is isomorphic to the intermediate Jacobian $\IJ(\cY)$.  
Thus we conclude
$$\IJ(\cX)=\JJ(C) \times E.$$



We examine sections of $\pi$ of low degree.
\begin{enumerate}
\item{lines on $\cY$;}
\item{conics on $\cY$ incident to $E$;}
\item{twisted cubics on $\cY$ incident to $E$ in two points.}
\end{enumerate}
The first family is indexed by $\IJ(\cY)$.
The second family has dimension three, thus {\em a priori} has a 
chance to dominate $\IJ(\cX)$.  However, the conics in $\cY$ fail to dominate
$\IJ(\cY)$, so the subfamily we're considering cannot possibly dominate 
$\IJ(\cX)$.  The family of twisted cubics in $\cY$ is birational to a 
Grassmann-bundle over the lines in $\cY$:  given $\ell \subset \cY$,
each subspace 
$$\ell \subset \bP^3 \subset \bP^5$$
intersects $\cY$ in the union of $\ell$ and a twisted cubic $R$.
Given $e_1,e_2 \in E$, consider
$$\mathrm{span}(\ell,e_1,e_2) \cap \cY = \ell \cup R,$$
which necessarily meets $E$ in $e_1$ and $e_2$.  
The curves produced in this way are parametrized by
$$\JJ(C) \times \mathrm{Sym}^2(E),$$
a $\bP^1$-bundle over $\JJ(C) \times E$.  

These one-parameter families of sections sweep out a distinguished
rational curve in $X$, the generic fiber of $\cX \ra \bP^1$.  
Fix a pair $(Q,R)$ consisting of
\begin{itemize}
\item{a quadric surface $E \subset Q \subset \mathrm{span}(Q)$;}
\item{a ruling $R\in \Pic(Q)$ such that $\cO_Q(R)|E=\cO_E(e_1+e_2)$.}
\end{itemize}
There is a natural bijection between such pairs and elements of
$\Pic^2(E)$.  Consider the singular quadric $C_{\ell}(Q) \subset \bP^5$ that
is a cone over $Q$ with vertex $\ell$.  The intersection of $C_{\ell}(Q)$
with a fiber $\cX_t$ of $\cX \ra \bP^1$ consists of $E$
plus a hyperplane section singular at $\ell \cap \cX_t$.

\section{Structure of height ten families}
\label{sect:test}

Let $Q\subset \bP^4$ be a 
smooth quadric threefold.  Consider a pencil of quadrics on $Q$
with smooth base locus $C$, a canonical curve of genus $5$.  Thus
we have a del Pezzo fibration
$$
\pi:\cX:=\mathrm{Bl}_C(Q) \ra \bP^1.
$$
We assume the fibers have at most one ordinary double point.  

\begin{prop}
$\pi$ has a family of sections parametrized by a principal 
homogeneous space over the Jacobian of $C$.
\end{prop}

\begin{proof}
Figure~\ref{fig1} lists numerical possibilities for the
spaces of sections, expressed
in terms of the geometry of the associated curve $R\subset Q$.
\begin{figure}
\begin{tabular}{ccc}
degree of $R$ & $\#(R\cap C)$ & number of parameters \\
\hline
  $0$         &  NA    &   $1$ \\
  $1$         &  $1$    &   $2$ \\
  $2$         &  $3$    &   $3$ \\
  $3$         &  $5$    &   $4$ \\
  $4$         &  $7$    &   $5$ \\
  $d$         &  $2d-1$    &   $d+1$
\end{tabular}
\caption{Parameter counts for spaces of sections}
\label{fig1}
\end{figure}

We focus on the latter case, as there are enough free parameters
for the family to dominate $\Pic^7(C)$.  

Thus we are left with an enumerative problem:
\begin{quote}
Fix $D \in \Pic^7(C)$;  then there exists a unique twisted quartic
$R \in Q$ such that $R \cap C \in |D|$.  
\end{quote}

Consider the linear series $|D|$.  Through each collection of seven
points $\Sigma=\{c_1,\ldots,c_7\}$ in linear general position
there passes a unique twisted quartic $R(\Sigma)$.  We are interested
in the locus
$$Y=\cup_{\Sigma \in |D|} R(\Sigma) \subset \bP^4.$$

The linear system $|D|$ realizes $C$ as a plane septic curve with 
$10$ double points $x_1,\ldots,x_{10}$.  
These depend on a total of $36-10-9=17$ parameters,
and $x_1,\ldots,x_{10}$ can be in general position.  Consider the image of
$\bP^2$ under the linear series of quartics through $x_1,\ldots,x_{10}$, 
i.e., the adjoint linear series.  The resulting surface $Y$ has 
degree $6$ in $\bP^4$.  

The quadric hypersurface $Q$ pulls back to a plane curve of degree eight
double at $x_1,\ldots,x_{10}$.  Of course, this contains $C$ and the residual curve
is a line in $\bP^2$.  Thus we find 
$$Y\cap Q= C \cup R$$
where $R$ is a twisted quartic in $\bP^4$.  

Now we are essentially done:  There exists a family of twisted quartic curves 
$$\begin{array}{rcl}
\cR & \ra &  \cX \\
\downarrow & & \\
W & &
\end{array}
$$
whose fibers each meet $C$ seven times.  The induced mapping
$$\mu: W \ra \Pic^7(C)$$
is birational.  However, $\Pic^7(C)$ is a principal homogeneous space
over the Jacobian of $C$.
\end{proof}

\begin{rema}
We summarize the geometric structure of the small-degree sections
listed in Figure~\ref{fig1}:  When $d=0$ these are just points on $C$, i.e.,
the constant sections.  For $d=1$, we have the lines incident to $C$ contained
in $Q$, a $\bP^1$-bundle over $C$.   When $d=2$, we have conics three-secant to
$C$, which are birationally parametrized by $\Sym^3(C)$;  indeed, take the 
plane spanned by three points on $C$ and intersect this with $Q$.  

For $d=3$, we have twisted cubics $R\subset Q$ meeting
$C$ in five points, which are birational to a $\bP^1$-bundle over 
$\Sym^3(C)$.  Given such a twisted cubic, let 
$H\subset \bP^4$ be the hypersurface it spans;  it contains a quadric surface
($Q\cap H$) and eight distinguished points ($C\cap H$) on that surface.  
Five of these points are on $R$;  three are not.  
Thus we have a mapping
$$\{\text{ twisted cubics meeting $C$ five times }\} \dashrightarrow \Sym^3(C).$$
The fiber over $p+q+r \in \Sym^3(C)$ consists of pairs
$$\{(H,R): H \supset \{p,q,r\}, (C\cap H) \setminus \{p,q,r\} \subset R \subset Q\cap H,$$
which is a double cover over $\bP^1$ via projection onto the first factor.  Indeed,
there are two twisted cubics passing through five generic points on a quadric surface.
However, given a pencil of three-planes containing a plane in $\bP^4$, two elements 
will be tangent to a quadric hypersurface $Q$.  Thus the double cover over $\bP^1$ 
is branched in two points, so the fibers of the mapping are rational curves.
\end{rema}

\begin{ques}
Where else does this periodic behavior occur for minimal quartic del Pezzo surfaces over
$F=k(t)$?  Can this be analyzed via moduli of vector bundles and Serre's construction?
\end{ques}

\section{Structure of families of height twelve}
\label{sect:twelve}

Recall the notation of Example~\ref{exam:heighttwelve}.  We have
two projections
$$\begin{array}{cc}
\cX & \subset  \bP(V^{\vee})=\bP(\cO_{\bP^1}(-1)^4 \oplus \cO_{\bP^1})  \stackrel{\pi_2}{\ra}  \bP^8\\
\quad \downarrow \scriptstyle{\pi}      &     \\
 \bP^1  &		
\end{array}
$$
and let $\cY$ denote the image of $\cX$ under $\pi_2$.  The canonical
section $\sigma(\bP^1)$ is contracted under $\pi_2$ to a double point $y\in \cY$.  
The image is a nodal Fano threefold of genus seven and degree twelve.  
See \cite{MukAJM} for concrete descriptions of {\em smooth} Fano varieties
of this type.

There is an alternate approach to $\cX$:  Project $\cX \ra \bP^1$ from the
distinguished section $\sigma(\bP^1)$
$$\widetilde{\cX} = \Bl_{\sigma(\bP^1)}(\cX) \subset \bP(\cO_{\bP^1}(-1)^{4}) \simeq \bP^1 \times \bP^3;$$
let $\cE$ denote the exceptional divisor.
The resulting fibration
$$\tilde{\pi}:\widetilde{\cX} \ra \bP^1$$
is a cubic surface fibration containing a constant line $\cE$. 
Now projection onto the second factor
$$\tilde{\pi}_2:\widetilde{\cX} \ra \bP^3$$
is birational and collapses $\cE$ onto a line $E\subset \bP^3$.  
The center of $\tilde{\pi}_2$ is the union of $E$ and a curve $C$,
of genus seven, admitting a line bundle $M$ with $h^0(M,C)=4$,
and $\deg(M)=8$, i.e., a $g^3_8$; the divisor $K_C-M$ is a $g^1_4$.
Note that $E$ and $C$ meet in four points $r_1,r_2,r_3,$ and $r_4$;
 the fibers of $\tilde{\pi}_2$
are mapped onto the pencil of cubic surfaces cutting out $E\cup C$.  
The birational mapping obtained by composing the inverse of $\tilde{\pi}_2$,
the blow up of $\cX$, and $\pi_2$
$$\bP^3 \dashrightarrow \cY \subset \bP^8$$
is induced by the linear series of quartics vanishing along $C\cup E$.  

Note that the intermediate Jacobian $\IJ(\cX)\simeq \JJ(C)$.  
This also has a Prym interpretation:  Projecting from $\cE$ fiberwise
allows us to realize $\widetilde{\cX}$ as a conic bundle
$$\widetilde{\cX} \ra \bP^1 \times \bP^1,$$
with discriminant a genus-eight curve $D$ of bidegree $(3,5)$.  If
$\tilde{D} \ra D$ is the cover arising from the conic bundle then
$$\IJ(\cX)\simeq \Prym(\tilde{D} \ra D).$$
This is an instance of the tetragonal construction of Donagi and Recillas
\cite{DonagiBAMS}.

We characterize sections of $\pi$ of small degree, using the geometry
sketched above:
\begin{enumerate}
\item{the canonical section $\sigma(\bP^1)$;}
\item{the sections arising from exceptional fibers over points of $C$,
which move in a one-parameter family;}
\item{the sections arising from secant lines to $C\subset \bP^3$,
which are parametrized by $\Sym^2(C)$;}
\item{the sections arising from conics five-secant to $C$,
which are parametrized by $\Sym^3(C)$;}
\item{the sections arising from twisted cubics meeting $C$ in eight points,
which are parametrized by $\Sym^4(C)$; we explain why
this is the case below;}
\item{the sections arising from eleven-secant quartic rational curves,
which are parametrized by $\Sym^5(C)$;}
\item{the sections arising from quintic rational curves
$14$-secant to $C$, which are parametrized by $\Sym^6(C)$;}
\item{the sections arising from sextic rational curves
$17$-secant to $C$, which are parametrized by $\Sym^7(C)$.}
\end{enumerate}

First the conic case:  each triple
of points on $C$ spans a plane, whose residual intersection 
with $C$ consists of five points that uniquely determine a 
five-section conic.

\

We do the cubic case.  Given such a section $R$ meeting
$C$ in $p_1,\ldots,p_8$, there is a three-dimensional vector space of
quadrics cutting out $R$, which induces a $g^2_8$ on $C$.  
Note that the variety of $g^2_8$'s on a 
genus seven curve is four-dimensional, reflecting the
fact that a $g^2_8$ is of the form $K_C-c_1-c_2-c_3-c_4$
for $c_1,c_2,c_3,c_4 \in C$.  This gives a morphism from the 
space of sections to $\Sym^4(C)$.

Conversely, suppose we are given generic points $c_1,c_2,c_3,c_4 \in C$.
Choose a twisted cubic $R'$ incident to $E$ at two points and containing $c_1,\ldots,c_4$.  (The curve $R'$ is a tri-section of our original del Pezzo
fibration.)
The union $R'\cup E \cup C$ has degree $12$ and arithmetic genus $15$, thus
$$\dim I_{R'\cup E\cup C}(4)=35-(48+1-15)=1$$
and the curve $R' \cup E \cup C$ sits on a quartic surface $S$.  
Note that this generically
has Picard lattice:
$$\begin{array}{c|cccc}
    & h  & C & E& R' \\
 \hline
h   & 4   & 8 & 1& 3 \\
C   & 8  & 12 & 4& 4 \\
E   & 1  & 4  & -2 & 2 \\ 
R'  & 3 &  4  & 2  & -2
\end{array}
$$  
Let $F$ be the residual to $R'\cup E$ in a complete intersection of a 
quadric surface and $S$; this is an elliptic quartic curve.  
Consider the union $C\cup F \cup E$;  we have
$$\begin{array}{rcl}
h^0(\cI_{C\cup F \cup E}(4)) & \ge  & 35-h^0(\cO_{C\cup F \cup E}(4))\\
& &= 35-(4\deg(C\cup F \cup E)+1-\sg(C\cup F\cup E))\\
& &=35-(52+1-20)=2.
\end{array}
$$
Let $R$ be residual to $C\cup F\cup E$ in this complete intersection.
Note that
$$C\cdot R=8, \quad R\cdot R=-2,$$
which gives the desired section.

\

We turn to the quartic case.  First we construct the morphism 
from the space of sections to $\Sym^5(C)$.  Suppose that
$R$ is a rational quartic curve in $\bP^3$
meeting $C \subset \bP^3$ in eleven points.  Choose a K3 surface
$S$ containing $C$ and $R$, with Picard lattice:
$$\begin{array}{c|ccc}
    & h  & C & R \\
 \hline
h   & 4   & 8 & 4 \\
C   & 8  & 12 & 11 \\
R   & 4  & 11 & -2 
\end{array}
$$
These exist as
$$\dim I_{C\cup R}(4) \ge 35-(48+1-17)=3.$$
Now $R$ is also contained in a quadric surface $Q$, generally
unique.  Let $R'$ denote the residual to $R$ in $S\cap Q$,
so that
$$h\cdot R'=4, \quad R'\cdot R' = -2, \quad C\cdot R'=5.$$
The intersection of $C$ with $R'$ gives five points on $C$.

We refer the reader to \cite[Sect.~8]{HTembed} for the inverse construction.

\

Now for the case of quintic rational curves:
Given six points $c_1,\ldots,c_6 \in C$, let $R'$ denote the unique
twisted cubic in $\bP^3$ through these points.  Consider the
pencil of quartic surfaces containing the union
$$C \cup_{c_1,\ldots,c_6} R',$$
each of which has the lattice polarization:
$$\begin{array}{r|ccc}
& h  &  C  & R' \\
\hline
h & 4  & 8 & 3 \\
C & 8 & 12 & 6 \\
R' & 3 & 6 & -2
\end{array}
$$
Consider the residual curve $R$ to $C\cup R'$ in this complete intersection;
it has genus zero, degree five, and meets $C$ in fourteen points.  
Thus $R$ corresponds to a section of $\pi$; these are parametrized
by elements of $\Sym^6(C)$, which may be interpreted as the theta divisor of
$\JJ(C)$.

\

Finally, we turn to the case of sextic rational curves.  
We construct the morphism to $\Sym^7(C)$.  Suppose that
$R$ is $17$-secant to $C$.  The union $C\cup R$ has
genus $23$ and degree $14$ so by Riemann-Roch
$$h^0(\cO_{C\cup R}(4))\ge 56+1-23=34$$
and thus is contained in a quartic surface $S$.  This has Picard
group:
$$\begin{array}{r|ccc}
& h  &  C  & R \\
\hline
h & 4  & 8 & 6 \\
C & 8 & 12 & 17 \\
R & 6 & 17 & -2
\end{array}
$$
Consider the elliptic fibration $C-h$; it has $R'=3h-R$ as a section
$$(C-h)\cdot R'=1, \quad C\cdot R'=7.$$
Geometrically, this arises from $R$ via residuation:  A rational
sextic space curve $R$ satisfies
$$h^0(I_R(3)) \ge 20- (18+1)=1$$
and thus lies in a cubic surface $T$.  Then $R'$ is the residual to $R$
in $T\cap S$.  The intersection of $R'$ with $C$ consists of seven points,
which gives us our point in $\Sym^7(C)$.  

This construction is invertible by the following result, proven in 
\cite{HTembed}:
\begin{quote}
Let $C \subset \bP^3$ denote a generic tetragonal curve of genus $7$ 
as above, and $c_1,\ldots,c_7 \in C$ generic points.  Then there
exists a {\em unique} sextic rational curve $R'\subset \bP^3$ 
and quartic K3 surface $S\subset \bP^3$ such that
$$R'\cap C= \{c_1,\ldots,c_7 \}$$
and 
$$C\cup R' \subset S.$$
\end{quote}

\begin{rema}
There is a straightforward construction of a family
of sections parametrized by a rationally connected fibration over $\JJ(C)$.
Fix $c_1,\ldots,c_7 \in C$ and let $R''$ denote a rational quartic
curve containing these points.  Note that the possible such $R''$
correspond to ruled quadric surfaces through $c_1,\ldots,c_7$, 
a double cover of $\bP(I_{c_1,\ldots,c_7}(2))$ branched over a 
quartic plane curve, i.e., a degree two del Pezzo surface.  
These curves yield five-sections of $\pi:\cX \ra \bP^1$.  Using
the argument for Brumer's theorem \cite{Brumer,Leep}, for each such 
odd-degree multisection there exists a family of sections, obtained
via a sequence of residuations, and parametrized by a rationally
connected variety. 
\end{rema}


We summarize the evidence for the existence of sections over generic fibrations of small height
collected so far.  

\

\centerline{
\begin{tabular}{llll}
height & Existence of sections? & Cross reference  \\
\hline
  0  & always  & Remark \ref{rema:heightzero} \\
  2  & --      & (see Example~\ref{exam:heighttwo})  \\
  4  & always  & Remark \ref{rema:heightfour} \\
  6  & always  & Remark \ref{rema:heightsix} \\
  8  & always & Remark \ref{rema:heighteight} \\
  10  & always  & Remark \ref{rema:heightten} and \S\ref{sect:test} \\
  12  & always  & Remark \ref{rema:heighttwelve}
\end{tabular}
}

\section{Fibrations of height fourteen to twenty}
\label{sect:heightlarger}

Let $\pi:\cX \ra \bP^1$ be a generic quartic del Pezzo fibration of height $14+2m, m=0,1,2,3$.
This corresponds to the following cases of Section~\ref{sect:explicit}:
\begin{itemize}
\item{height $14$: Example~\ref{exam:heightfourteen} of Case 4;}
\item{height $16$: the $n=0$ instance of the even part of Case 3;}
\item{height $18$: the $n=0$ instance of the odd part of Case 2;}
\item{height $20$: the $n=0$ instance of the odd part of Case 1.}
\end{itemize}
Then we have
$$
\begin{array}{ccccc}
\cX & \subset & \bP^1 \times \bP^{7-m} & \stackrel{\pi_2}{\ra} & \bP^{7-m} \\
    &         &  {\scriptstyle \pi} \downarrow \quad &         &           \\
    &         & \bP^1  &				&    
\end{array}
$$
and we write $\cY=\pi_2(\cX)$.  
We have:
\begin{itemize}
\item{$\cY$ is anticanonically embedded in $\bP^{7-m}$ of degree $10-2m$;}
\item{$\pi_2:\cX \ra \cY$ is small, generically contracting 
$2^{m+1}$ sections of $\pi$ to nodes of $\cY$.}
\end{itemize}
Indeed, the anticanonical divisor of $\cX$ (as computed via adjunction)
coincides with $\pi_2^*\cO_{\bP^{7-m}}(1)$.  
A direct computation with Bezout's theorem shows that the number of sections
contracted by $\pi_2$ is as claimed.  
(Over non-closed fields, there is no reason for these to be defined over the ground
field.)

In the following sections, we make the geometry even more explicit.  

\subsection*{Height fourteen}
Recall the notation from Example~\ref{exam:heightfourteen}.
The two distinguished sections are denoted $\sigma_1,\sigma_2:\bP^1 \ra \cX$. 
Projection onto the second factor gives
$$\begin{array}{ccccc}
\cY & \subset & \pi_2(D) & \subset & \pi_2(\bP(V^{\vee})) \\
	   &         &  ||      &         &    ||                \\
       &    & W       &         &   \Sigma.
\end{array}
$$
Note that $\Sigma$ is a cone over a Segre threefold $\bP^1 \times \bP^2 \subset \bP^5$,
with vertex $\ell$, a line.  
The equations of $\Sigma$ are 
given by the $2\times 2$ minors of the matrix:
$$\left( \begin{matrix} u & v & w \\
		        x & y & z \end{matrix} \right);$$
we express
$$\ell=\{u=v=w=x=y=z=0 \} \subset \Sigma.$$

We may interpret $W$ as follows:  Fix a fiber 
$\{\mathrm{pt.}\} \times \bP^2$ in the Segre threefold, yielding a subspace
$$\Lambda\simeq \bP^4 \subset \Sigma, \quad \Lambda \supset \ell.$$  
Write
$$\Lambda=\{u=v=w=0 \} \subset \Sigma$$
and 
$$Q'=\{q_1(u,v,w)+u l_1(x,y,z,s,t)+v l_2(x,y,z,s,t) + w l_3(x,y,z,s,t)=0 \},$$
where $q_1$ is quadratic and the $l_i$ are linear. 
Note that $Q'$
is necessarily singular along some line $\ell' \subset \Lambda$, 
and thus has rank at most six;  
in suitable coordinates, we have
$$Q'=\{u m_1+v m_2 + w m_3=0 \},$$
where the $m_i$ are linear.
Let $W$ denote the residual to $\Lambda$ in $Q'\cap \Sigma$, which
is contained in the additional quadric hypersurface
$$Q''=\{x m_1 + y m_2 + z m_3 =0\}.$$
Suppose that $\Lambda'\simeq \bP^4$ is a cone over a generic fiber 
$\{\mathrm{pt.}'\} \times \bP^2$ in the Segre threefold;  then
$Q'$ and $Q''$ cut out the same quadric hypersurface on $\Lambda'$, i.e.,
$$Q'\cap \Lambda'=Q'' \cap \Lambda'.$$

We claim that $W$ is a linear section of the Grassmannian $\Gr(2,5) \subset \bP^9$.
Indeed, let $\xi_{ij},1 \le i<j \le 5$ denote the Pl\"ucker coordinates
on the Grassmannian and consider the linear section $\xi_{14}=0$.  
The Pl\"ucker relations specialize to the $2\times 2$ minors of
$$\left( \begin{matrix} \xi_{12} & \xi_{13} & -\xi_{15} \\
			\xi_{24} & \xi_{34} & \xi_{45} \end{matrix} \right)$$
as well as 
$$\xi_{12}\xi_{35}-\xi_{13}\xi_{25}+\xi_{15}\xi_{23}=0 \quad
-\xi_{23}\xi_{45}+\xi_{24}\xi_{35}-\xi_{25}\xi_{34}=0.$$
Setting 
$$u=\xi_{12},v=\xi_{13},w=-\xi_{15},x=\xi_{24},y=\xi_{34},
z=\xi_{45}$$
and $m_1=\xi_{35},m_2=-\xi_{25},$ and $m_3=-\xi_{23}$,
we obtain the equations for $W$ given above.  

In particular, $W$ has degree five.  Furthermore, a direct
calculation shows that $W$ has ordinary double points along $\ell$.
This reflects the fact that the birational morphism $\pi_2:D\ra W$ 
collapses $\ell \times \bP^1$ to $\ell$,
and any small contraction of a smooth variety is singular. 

The variety $\cY$ is the intersection of $W$ with a generic quadric $Q$, and thus has degree ten
and two ordinary singularities 
$$\{y_1,y_2 \}=\ell \cap Q= \{\pi_2(\sigma_1(\bP^1)),\pi_2(\sigma_2(\bP^1))\}.$$
The dualizing sheaf $\omega_{\cY}=\cO_{\cY}(-1)$, so $\cY$ is a Fano variety with
ordinary double points, and thus a degeneration of a smooth Fano variety
of degree ten and genus six \cite{Namikawa}.  
Degenerations with {\em one} node are studied in \cite{DIM2}.

We describe the quartic del Pezzo fibration in terms of the geometry of $\cY \subset \bP^7$.  The fibers are obtained by intersecting the subspaces $\Lambda'$ specified
above (associated with the fibers of $\bP^1 \times \bP^2 \ra \bP^1$)
with $Q$ and $Q'$;  note that $Q'\equiv Q''\pmod{\Lambda'}$.

\subsection*{Height sixteen}
In this case $\cY \subset \bP^6$ is a complete intersection of three quadrics with four nodes.
These have expected dimension
$$\dim \Gr(3,\Gamma(\cO_{\bP^6}(2))) - \dim \PGL_7 - 4 = 23,$$
which is equal to the expected dimension of the space of quartic del Pezzo fibrations of height sixteen.

\subsection*{Height eighteen}
Here $\cY \subset \bP^5$ is a complete intersection of a quadric and cubic with eight nodes.  
These have expected dimension $26$,
which is equal to the expected dimension of the space of quartic del Pezzo fibrations of height eighteen.

\subsection*{Height twenty}
Here $\cY \subset \bP^4$ is a quartic threefold with sixteen nodes; however, these
have expected dimension $30$, greater than the expected dimension ($\frac{3}{2}h(\cX)-1=29$)
of the quartic
del Pezzo fibrations of height twenty.

Consider the Grassmannian $\Gr(2,\Gamma(\cO_{\bP^4}(2)))$ parametrizing 
quartic del Pezzo surfaces.  A generic conic curve in this Grassmannian corresponds to the
rulings of a nonsingular quadric surface in $\bP(\Gamma(\cO_{\bP^4}(2)))$.  
These can be expressed in equations as
$$P_1s+Q_1t=P_2s+Q_2t=0$$
where $P_1,Q_1,P_2,Q_2 \in \Gamma(\cO_{\bP^4}(2))$.  In other words
$$\cX \subset \bP^1 \times \bP^4$$
as a complete intersection of two hypersurfaces of bidegree $(1,2)$.  
These depend on 
$$2(2\times 15-2) - (3+24)=29$$
parameters.  
Let $\sigma_1,\ldots,\sigma_{16}: \bP^1 \ra \cX$ denote the constant sections, corresponding to the solutions 
$$\{y_1,\ldots,y_{16}\}=\{P_1=Q_1=P_2=Q_2=0\}.$$
Projection onto the second factor gives the quartic hypersurface
$$\cY=\{P_1Q_2-Q_1P_2=0\} \subset \bP^4,$$
with nodes at $y_1,\ldots,y_{16}$.  

There is a second del Pezzo fibration
$$\cY \dashrightarrow \bP^1$$
constructed as follows.  Fix a fiber $\cX_t =\pi^{-1}(t)$ and its image 
$\cX_t \subset \cY$.  The linear series of quadrics vanishing at $\cX_t$ gives
a mapping $\pi'$, whose fibers are also del Pezzo surfaces of degree four.  
This is a morphism where $\cX_t \subset \cY$ is Cartier, i.e., away 
from $y_1,\ldots,y_{16}$.  
There exists a small resolution $\cX'\ra \cY$ such that the induced
$$\pi':\cX' \ra \bP^1$$
is a morphism.  We obtain $\cX'$ from $\cX$ by flopping each of 
$\sigma_1,\ldots,\sigma_{16}$.  
Finally, let $\tilde{Y}$ denote the blow up of $\cY$ at each of 
the $y_i$, which coincides with the closure of the graph of the 
birational mapping $\cX \dashrightarrow \cX'$;  the induced morphism
$$(\pi,\pi'):\tilde{Y} \ra \bP^1 \times \bP^1$$
has genus five canonical curves as fibers.  

Projecting from one of the $\sigma_j$, we obtain a conic bundle
$$
\cW \ra \bP^1\times \bP^1
$$
branched along a curve $D$ of bidegree $(5,5)$, of genus $16$.  
The Prym variety therefore has dimension $15$; see also \cite{Cheltsov},
where the non-rationality of a generic $\cX$ is established.

\bibliographystyle{alpha}
\bibliography{delPezzo}

\end{document}